\documentclass[12pt]{article}
\usepackage{amssymb,amsmath,amscd,epsfig,amsthm}

\newcommand{\Stf}{\mathop{\rm Stab}\nolimits_F}
\newcommand{\Stg}{\mathop{\rm Stab}\nolimits_G}

\newtheorem{cor}{Corollary}
\newtheorem{prop}{Proposition}
\newtheorem{lemma}{Lemma}
\newtheorem{defin}{Definition}

\newcounter{thmcounter}

\newtheorem{theorem}[thmcounter]{Theorem}

\begin{document}
\title{Some graphs related to Thompson's group $F$}
\author{Dmytro Savchuk\thanks{Supported by NSF grants DMS-0600975 and
DMS-0456185}\\ Texas A\&M University}

\maketitle

\abstract{The Schreier graphs of Thompson's group $F$ with respect
to the stabilizer of $\frac12$ and generators $x_0$ and $x_1$, and
of its unitary representation in $L_2([0,1])$ induced by the
standard action on the interval $[0,1]$ are explicitly described.
The coamenability of the stabilizers of any finite set of dyadic
rational numbers is established. The induced subgraph of the right
Cayley graph of the positive monoid of $F$ containing all the
vertices of the form $x_nv$, where $n\geq0$ and $v$ is any word over
the alphabet $\{x_0, x_1\}$, is constructed. It is proved that the
latter graph is non-amenable.}

\section*{Introduction}

Thompson's group $F$ was discovered by Richard Thompson in 1965. A
lot of fascinating properties of this group were discovered later
on, many of which are surveyed nicely in~\cite{intro_tomp}. It is a
finitely presented torsion free group. One of the most intriguing
open questions about this group is whether $F$ is amenable.
Originally this question was asked by Geoghegan in 1979 (see p.549
of~\cite{gersten_s:comb_group_theory}) and since then dozens of
papers were in some extent devoted to it. It was shown
in~\cite{brin_s:piecewise} that $F$ does not contain a nonabelian
free subgroup and in~\cite{intro_tomp} that it is not elementary
amenable. So the question of amenability of $F$ is particularly
important because $F$ would be an example of a group given by a
balanced presentation (two generators and two relators) of either
amenable, but not elementary amenable group (the first finitely
presented example was constructed by R.Grigorchuk
in~\cite{grigorch:example}), or non-amenable group, which does not
contain a nonabelian free subgroup (the first finitely presented
example of this type was constructed by Ol$'$shanskii and Sapir
in~\cite{olshanskii_s:non-amenable}).

The study of the Schreier graphs of $F$ was also partially inspired
by the question of amenability of $F$. In particular, if any
Schreier graph with respect to any subgroup is non-amenable the
whole group $F$ would be non-amenable. Unfortunately, all Schreier
graphs we describe here are amenable which does not give any
information about the amenability of $F$. But the knowledge about
the structure of Schreier graphs provides some additional
information about $F$ itself.

It happens that the described Schreier graph of the action of $F$ on
the set of dyadic rational numbers on the interval $(0,1)$ is
closely related to the unitary representation of $F$ in the space
$B(L_2([0,1]))$ of all bounded linear operators on $L_2([0,1])$. It
reflects (modulo a finite part) the dynamics of $F$ on the Haar
wavelet basis in $L_2([0,1])$. We define the Schreier graph of the
group action on the Hilbert space with respect to some basis and
make this connection precise.

R.~Grigorchuk and S.~Stepin in~\cite{grigorch_s:amenab} reduced the
question of amenability of $F$ to the right amenability of the
positive monoid $P$ of $F$. Moreover, the amenability of $F$ is
equivalent to the amenability of the induced subgraph $\Gamma_P$ of
the Cayley graph $\Gamma_F$ of $F$ with respect to generating set
$\{x_0,x_1\}$ containing the positive monoid $P$. We construct the
induced subgraph $\Gamma_S$ of $\Gamma_F$ containing all the
vertices of the form $x_nv$ for $n\geq 0, v\in \{x_0,x_1\}^*$ and
prove that this graph is non-amenable. In this construction we use
the realization of the elements of the positive monoid of $F$ as
binary rooted forests. The existence of this representation was
originally noted by K.~Brown and developed by J.~Belk
in~\cite{belk:phd} and Z.~\v Suni\'c in~\cite{sunic:tamari}. It was
also used by J.~Donelly in~\cite{donnelly:ruinous} to construct an
equivalent condition for amenability of $F$.

The structure of the paper is as follows. In
Section~\ref{sec_thompson_group} the definition and the basic facts
about Thompson's group are given. Section~\ref{sec_schreier}
contains the description of the Schreier graph of the action of $F$
on the set of dyadic rational numbers from the interval $(0,1)$. The
coamenability of the stabilizers of any finite set of dyadic
rational numbers is shown in Section~\ref{sec_coamenab}. The
Schreier graph of the action of $F$ on $L_2([0,1])$ is constructed
in Section~\ref{sec_action_of_L2}. The last Section~\ref{sec_cayley}
contains a description of the subgraph $\Gamma_S$ of $\Gamma_P$ and
the proof that $\Gamma_S$ is non-amenable.

The author expresses warm gratitude to Rostislav Grigorchuk for
valuable comments and bringing his attention to Thompson's group,
and to Zoran \v Suni\'c, who has pointed to the connection with
forest diagrams, which simplified the proofs in the last section.

\section{Thompson's group}
\label{sec_thompson_group}

\begin{defin}
The \emph{Thompson's group} $F$ is the group of all strictly
increasing piecewise linear homeomorphisms from the closed unit
interval $[0,1]$ to itself that are differentiable everywhere except
at finitely many dyadic rational numbers and such that on the
intervals of differentiability the derivatives are integer powers of
$2$. The group operation is superposition of homeomorphisms.
\end{defin}

Basic facts about this group can be found in the survey
paper~\cite{intro_tomp}. In particular, it is proved that $F$ is
generated by two homeomorphisms $x_0$ and $x_1$ given by
$$x_0(t)=\left\{
\begin{array}{ll}
\frac t2,&0\leq t\leq\frac12,\\
t-\frac14,&\frac12\leq t\leq\frac34,\\
2t-1,&\frac34\leq t\leq1,\\
\end{array} \right.
\qquad x_1(t)=\left\{
\begin{array}{ll}
t,&0\leq t\leq\frac12,\\
\frac t2+\frac14,&\frac12\leq t\leq\frac34,\\
t-\frac18,&\frac34\leq t\leq\frac78,\\
2t-1,&\frac78\leq t\leq1.\\
\end{array} \right.
$$

The graphs of $x_0$ and $x_1$ are displayed in
Figure~\ref{fig_gens}.
\begin{figure}[h]
\begin{center}
\includegraphics{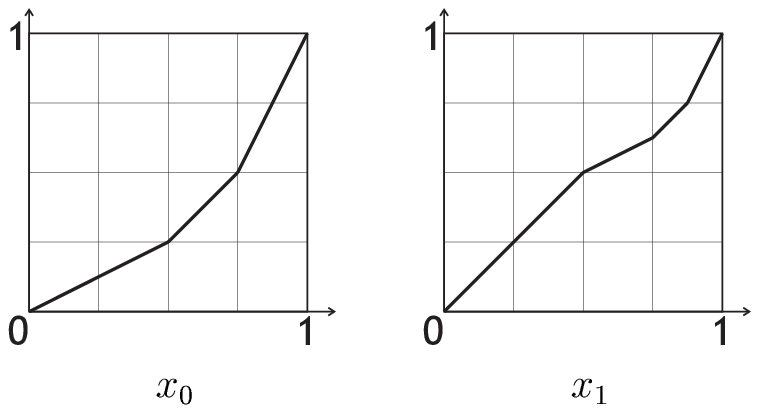}
\end{center}
\caption{Generators of $F$\label{fig_gens}}
\end{figure}

Throughout the paper we will follow the following conventions. For
any two elements $f$, $g$ of $F$ and any $x\in[0,1]$
\begin{equation}
\label{eqn_conv}
(fg)(x)=g(f(x)),\quad f^g=gfg^{-1}.
\end{equation}

With respect to the generating set $\{x_0,x_1\}$ $F$ is finitely
presented. But for some applications it is more convenient to
consider an infinite generating set $\{x_0,x_1,x_2,\ldots\}$, where
\[x_n=(x_1)^{x_0^{n-1}}.\]
With respect to this generating set (and with respect to
convention~\eqref{eqn_conv}) $F$ has a nice presentation
\begin{equation}
\label{eqn_presF}
F\cong\langle x_0,x_1,x_2,\ldots\ |\ x_kx_n=x_{n+1}x_k,\ 0\leq
k<n\rangle.
\end{equation}

\section{The Schreier graph of the action of $F$ on the set of dyadic rational numbers}
\label{sec_schreier}

Let $G$ be a group generated by a finite generating set $S$ acting
on the set $M$. The \emph{Schreier graph} $\Gamma(G,S,M)$ of the
action of $G$ on $M$ with respect to the generating set $S$ is an
oriented labelled graph defined as follows. The set of vertices of
$\Gamma(G,S,M)$ is $M$ and there is an arrow from $x\in M$ to $y\in
M$ labelled by $s\in S$ if and only if $x^s=y$.

For any subgroup $H$ of $G$, the group $G$ acts on the right cosets
in $G/H$ by right multiplication. The corresponding Schreier graph
$\Gamma(G,S,G/H)$ is denoted as $\Gamma(G,S,H)$ or just
$\Gamma(G,H)$ if the generating set is clear from the context.

Conversely, if $G$ acts on $M$ transitively, then $\Gamma(G,S,M)$ is
canonically isomorphic to $\Gamma(G,S,\Stg(x))$ for any $x\in M$,
where the vertex $y\in M$ in $\Gamma(G,S,M)$ corresponds to the
coset from $G/\Stg(x)$ consisting of all elements of  $G$ that move
$x$ to $y$.

Consider the subgroup $\Stf(\frac12)$ of $F$ consisting of all
elements of $F$ that fix $\frac12$. There is a natural isomorphism
$\psi:\Stf(\frac12)\to F\times F$ given by
\begin{equation}
\label{isom} \Stf\Bigl(\frac12\Bigr)\ni f(t)\stackrel{\psi}{\longmapsto} \left(2f\Bigl(\frac
t2\Bigr), 2\Bigl(f\Bigl(\frac{t+1}2\Bigr)-1\Bigr)\right)\in F\times F.
\end{equation}
This group was studied in~\cite{burillo:quasi-isom}, where it was
shown that it embeds into $F$ quasi-isometrically.

The Schreier graph $\Gamma(F,\{x_0,x_1\},\Stf(\frac12))$ coincides
with the Schreier graph of the action of $F$ on the orbit of
$\frac12$. Let $D$ be the set of all dyadic rational numbers from
the interval $(0,1)$. It is known that $F$ acts transitively on $D$
(which follows also from the next proposition). Therefore the latter
graph coincides with the Schreier graph $\Gamma(F,\{x_0,x_1\},D)$.

\begin{prop}
\label{shreier} The Schreier graph $\Gamma(F,\{x_0,x_1\},D)$ has the following structure (dashed arrows are
labelled by $x_0$ and solid arrows by $x_1$)

\begin{center}
\epsfig{file=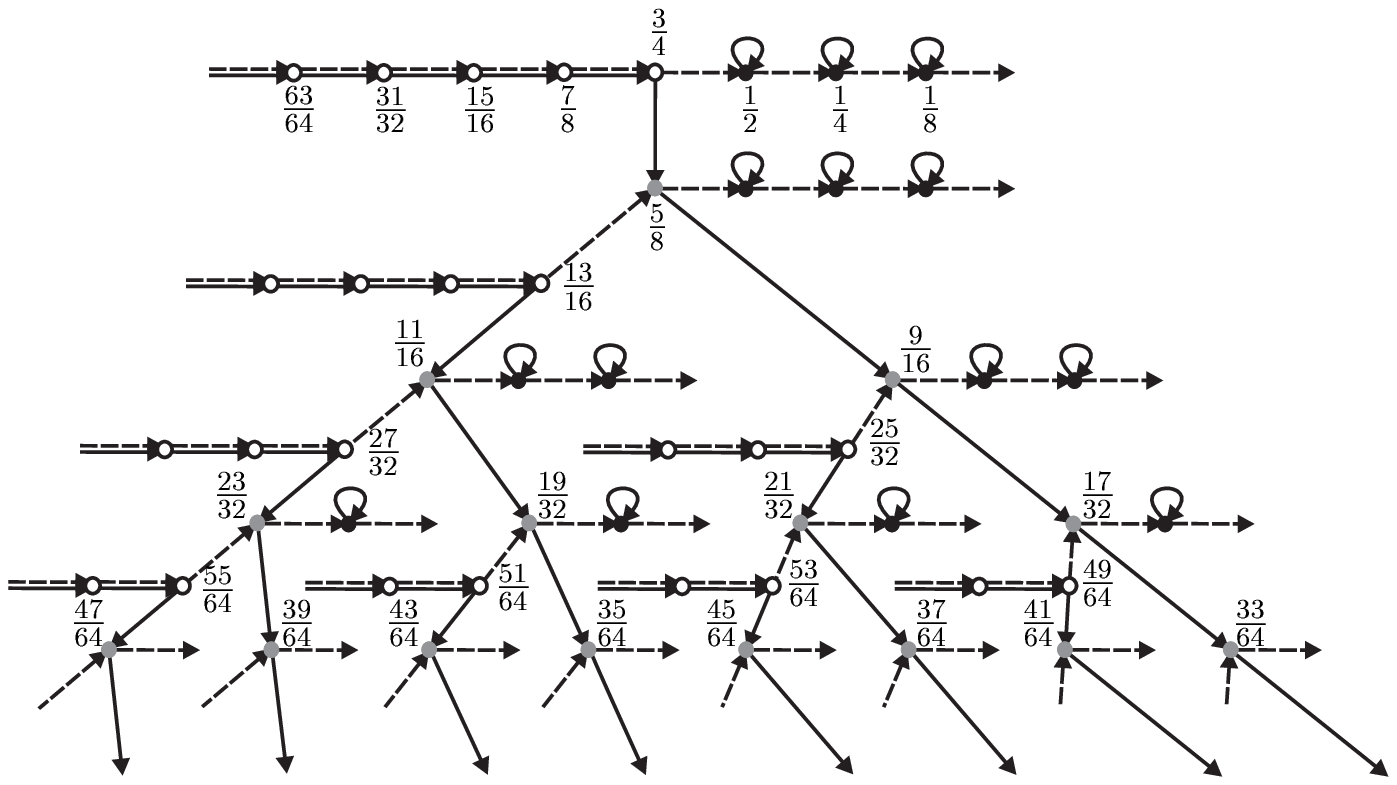,height=160pt}
\end{center}
\end{prop}

\begin{proof}
Define the following subsets of $D$.
$$A_n=\Bigl\{\frac k{2^n}\ \bigl|\ k\ \mathrm{is\ odd}\Bigr\}\cap\Bigl(\frac12,\frac34\Bigr), n\geq3$$
$$B_n=\Bigl\{\frac k{2^n}\ \bigl|\ k\ \mathrm{is\ odd}\Bigr\}\cap\Bigl(\frac34,\frac78\Bigr), n\geq4$$
$$C_n=A_n\cap\Bigl(\frac12,\frac58\Bigr),\ \  D_n=A_n\cap\Bigl(\frac58,\frac34\Bigr), n\geq4$$
On the graph above, $A_n$ represents the $(n-3)$-rd level of the
gray vertices in the binary tree; $B_n$ is the set of the white
vertices between levels $n-4$ and $n-3$ of the tree, adjacent to 2
gray vertices; $C_n$ and $D_n$ are the sets of the gray vertices of
the $(n-3)$-rd level having gray and white neighbors above
respectively.

\begin{figure}[h]
\begin{center}
\includegraphics{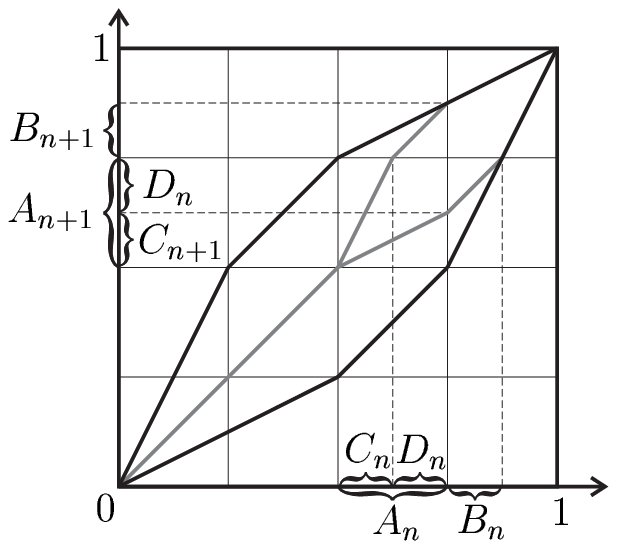}
\end{center}
\caption{Dynamics of $x_0$ and $x_1$\label{fig_proof}}
\end{figure}

Now we compute the action of $F$ on this subsets (see
Figure~\ref{fig_proof}). We have $x_0^{-1}(A_n)=B_{n+1}$,
$x_1(B_n)=D_n$, $x_1(A_n)=C_{n+1}$, hence
$(x_0^{-1}x_1)(A_n)=D_{n+1}$ and $(x_0^{-1}x_1)(A_n)\cup
x_1(A_n)=A_{n+1}$. Furthermore, for any set $A\subset\mathbb R$
denote $\alpha A+\beta=\{\alpha a+\beta: a\in A\}$. Then
$x_0^k(A_n)=x_0^kx_1(A_n)=2^{-k+1}(A_n-\frac14)$ for $k\geq1$. This
corresponds to the rays with the black vertices sticking out to the
right from the gray ones. On the other hand since the actions of
$x_0^{-1}$ and $x_1^{-1}$ on $[\frac34,1]$ coincide, for any element
$f$ of length $k\geq0$ from the monoid generated by $x_0^{-1}$ and
$x_1^{-1}$ we have $f(B_n)=1-2^{-k}(1-B_n)$.  This corresponds to
the rays with white vertices. There is one more geodesic line in the
graph corresponding to $\frac12$ which completes the picture.
\end{proof}

This graph gives alternative proofs of the following well-known
facts.

\begin{cor}
The subsemigroup of $F$ generated by $x_1$ and $x_0^{-1}x_1$ is
free.
\end{cor}

\begin{cor}
\label{trans}
\begin{itemize}
\item[{\rm (a)}] Thompson's group $F$ acts transitively on the set $D$ of all dyadic rationals from the
interval $(0,1)$.
\item[{\rm (b)}] $\Stf(\frac12)$ acts transitively on the sets of dyadic rationals from the intervals
$(0,\frac12)$ and $(\frac12,1)$.
\end{itemize}
\end{cor}
\begin{proof}
Part (a) follows immediately from the structure of the Schreier
graph $F/\Stf(\frac12)$. Part (b) is a consequence of part (a) and
the isomorphism~\eqref{isom}.
\end{proof}

\begin{prop}
The subgroup $\Stf(\frac12)$ is a maximal subgroup in $F$.
\end{prop}
\begin{proof}
Let $f$ be any element from $F\setminus\Stf(\frac12)$. Then for any
$g\in F$ we show that $g\in\langle\Stf(\frac12), f\rangle$. Let $g$
be an arbitrary element in $F$ that does not stabilize $\frac12$.

Denote $u=f(\frac12)$ and $v=g(\frac12)$. Without loss of generality
we may assume $u<\frac12$. Then by transitivity from
Corollary~\ref{trans}(b) there exists $h\in\Stf(\frac12)$ such that
either $h(f(\frac12))=v$ or $h(f^{-1}(\frac12))=v$ depending on
whether $v<\frac12$ or $v>\frac12$. In any case the element $\tilde
f=fh$ (or $\tilde f=f^{-1}h$) belongs to $\langle\Stf(\frac12),
f\rangle$ and satisfies $\tilde f(\frac12)=v$.

Now for $\tilde h=g\tilde f^{-1}$ we have $\tilde h(\frac12)=\tilde
f^{-1}(g(\frac12))=\tilde f^{-1}(v)=\frac12$. Thus $\tilde h\in
\Stf(\frac12)$ and $g=\tilde h\tilde f\in\langle\Stf(\frac12),
f\rangle$.
\end{proof}

Proposition~\ref{shreier} also yields a bound on the length of an
element. Namely, if the graph of an element $f\in F$ passes through
the point $(a,b)$ for some dyadic rational numbers $a$ and $b$, then
the length of $f$ with respect to the generating set $\{x_0,x_1\}$
is not smaller than the combinatorial distance between $a$ and $b$
in the graph $\Gamma(F,\{x_0,x_1\},D)$.

Estimates similar in spirit (also based on the properties of graph
of an element, but in a different realization of $F$) were used by
J.Burillo in~\cite{burillo:quasi-isom} to show that $\Stf(\frac12)$
quasi-isometrically embeds into $F$.

\vspace{.5cm}

\section{Coamenability of stabilizers of several dyadic
rationals}
\label{sec_coamenab}

In this section we show that for any finite subset
$\{d_1,\ldots,d_n\}$ of dyadic rationals the Schreier graphs of $F$
with respect to $\Stf(d_1,\ldots,d_n)$ is amenable, which,
unfortunately, does not give any information about amenability of
$F$.

First we recall the definition of an amenable graph.

\begin{defin}
Given an infinite graph $\Gamma=(V,E)$ of bounded degree \emph{the
Cheeger constant} $h(\Gamma)$ is defined as follows
$$h(\Gamma)=\inf_S\frac{|\partial S|}{|S|},$$
where $S$ runs over all nonempty finite subsets of $V$, and
$\partial S$, the boundary of $S$, consists of all vertices of
$V\setminus S$ that have a neighbor in $S$.
\end{defin}

\begin{defin}
The graph $\Gamma$ is called \emph{amenable} if $h(\Gamma)=0$.
\end{defin}

\begin{defin}
A subgroup $H$ of a group $G$ is called coamenable in $G$ if the
Schreier graph $\Gamma(G,H)$ is amenable.
\end{defin}

Note, that coamenability of a subgroup does not depend on the
generating set of $G$. This follows easily from Gromov's doubling
condition (see Theorem~\ref{thm_gromov} in
Section~\ref{sec_cayley}).

\begin{prop}
Let $\{d_1,\ldots,d_n\}\subset D$ be any finite subset of dyadic
rationals. Then the subgroup $\Stf(d_1,\ldots,d_n)$ of $F$
consisting of all elements stabilizing all the $d_i$'s is coamenable
in $F$.
\end{prop}

\begin{proof}
First, we describe the structure of the Schreier graph
$\Gamma(F,\{x_0,x_1\},\Stf(d_1,\ldots,d_n))$, $d_1<d_2<\cdots<d_n$.
Analogously to the singleton case there is a one-to-one
correspondence between cosets from $F/\Stf(d_1,\ldots,d_n)$ and all
strictly increasing $n$-tuples of dyadic rationals. This follows
from the fact that $F$ acts transitively on the latter set
(see~\cite{intro_tomp}). There is an edge labelled by
$s\in\{x_0,x_1\}$ from the coset $(d_1^\prime,\ldots,d_n^\prime)$ to
the coset $(d_1^{\prime\prime},\ldots,d_n^{\prime\prime})$ if and
only if $s(d_i^\prime)=d_i^{\prime\prime}$ for every $i$.

Geometrically one can interpret this in the following way. Consider
a disjoint union of $n$ copies of
$\Gamma(F,\{x_0,x_1\},\Stf(\frac12))$ (a layer for each $d_i$). Then
the coset $(d_1^\prime,\ldots,d_n^\prime)$ of
$F/\Stf(d_1,\ldots,d_n)$ can be represented by the path joining
$d_i^\prime$ vertex on the $i$-th layer with $d_{i+1}^\prime$ vertex
on the $(i+1)$-th layer (see Figure~\ref{fig_layers}). The action of
the generators on the set of such paths is induced by the
independent actions of the generators on the layers.

\begin{figure}[h]
\begin{center}
\epsfig{file=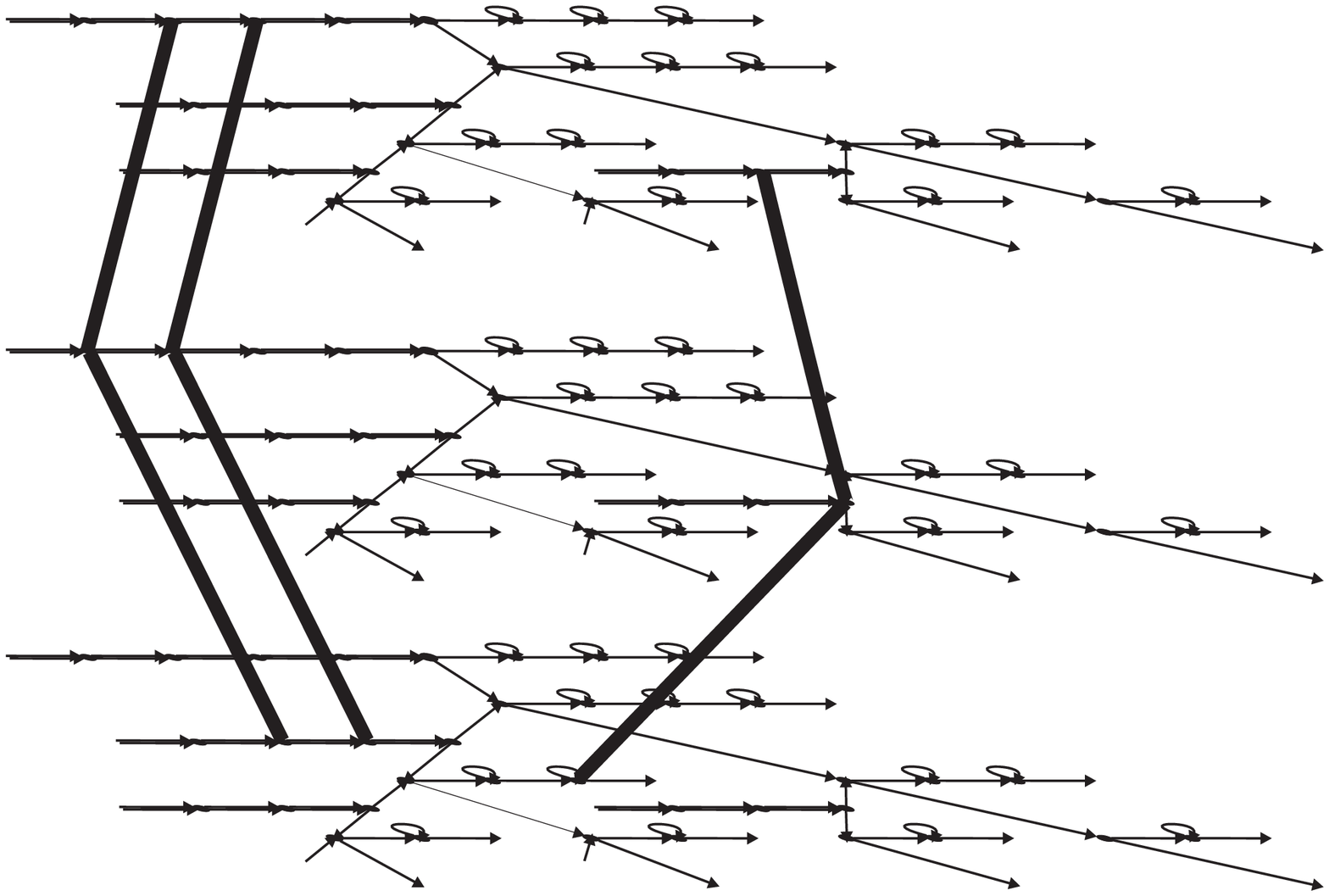,height=180pt}
\caption{Cosets in $F/\Stf(d_1,d_2,d_3)$\label{fig_layers}}
\end{center}
\end{figure}

Now define
$$E_i=\left(\frac1{2^{i+n}},\frac1{2^{i+n-1}},\ldots,\frac1{2^{i+1}}\right)\in F/\Stf(d_1,\ldots,d_n)$$
and
$$S_m=\left\{E_i\ \bigl|\ 1\leq i\leq m\right\}.$$

Since $x_1(E_i)=E_i$ and $x_0(E_i)=E_{i+1}$ we have that the
boundary $\partial S_m=\{E_0,E_{m+1}\}$ and
$$\lim_{m\to\infty}\frac{|\partial S_m|}{|S_m|}=\lim_{m\to\infty}\frac2m=0.$$

Thus $h(\Gamma(F,\{x_0,x_1\},\Stf(d_1,\ldots,d_n)))=0$ and
$\Stf(d_1,\ldots,d_n)$ is coamenable in $F$.
\end{proof}

The amenability of the action of $F$ on the set of dyadic rational
numbers and on the set of the ordered tuples of dyadic rational
numbers was also noted independently by N.~Monod and Y.~Glasner
(private communication).

\section{The Schreier graph of the action of $F$ on $L_2([0,1])$}
\label{sec_action_of_L2}
\noindent There is a natural unitary representation of Thompson's
group $F$ in the space $\mathcal B(L_2([0,1]))$ of all bounded
linear operators on $L_2([0,1])$. For $g\in F$ and $f\in L_2([0,1])$
define
$$(\pi_gf)(x)=\sqrt\frac{dg(x)}{dx}f(g^{-1}x).$$

For our purposes it is convenient to consider this action with
respect to the orthonormal Haar wavelet basis
$B=\{h^{(0)},h^{(i)}_j, i\geq0, j=1\ldots 2^i\}$ in $L_2([0,1])$,
where $h^{(0)}(x)\equiv1$ and

\begin{center}
\begin{tabular}{ll}
$h^{(0)}(x)\equiv1$, & $h^{(0)}_1(x)=\left\{
\begin{array}{l}
-1, x<\frac12,\\
1, x\geq\frac12, \\
\end{array}
\right.$
\end{tabular}
\end{center}

$$h^{(i)}_j(x)=\left\{
\begin{array}{l}
-2^{\frac i2}, \frac{j-1}{2^i}\leq x<\frac{j-1}{2^i}+\frac1{2^{i+1}},\\
2^{\frac i2}, \frac{j-1}{2^i}+\frac1{2^{i+1}}\leq x\leq\frac{j}{2^i}, \\
0, x\notin [\frac{j-1}{2^i},\frac{j}{2^i}].
\end{array}
\right.$$

This basis has first appeared in $1910$ in the paper of
Haar~\cite{haar:basis} and plays an important role in the wavelet
theory (see, for
example,~\cite{daubechies:wavelets,walter_s:wavelets}).

The convenience of using this basis for us comes from the following
fact. Each of the generators $x_0$ and $x_1$ acts on each of the
basis functions $h^{(i)}_j$ for $i\geq3$ linearly on the support of
$h^{(i)}_j$, so that the image also belongs to $B$. More precisely,
straightforward computations yield\vspace{.3cm}
\begin{equation}
\label{eqn_actionL2}
\begin{array}{l}
\vspace{.2cm} \pi_{x_0}h^{(i)}_j=h^{(i+1)}_j, \quad i\geq 1, \quad
1\leq
j\leq2^{i-1},\\

\vspace{.2cm} \pi_{x_0}h^{(i)}_j=h^{(i)}_{j-2^{i-2}},\quad  i\geq
2,\quad
2^{i-1}+1\leq j\leq2^{i-1}+2^{i-2},\\

\vspace{.5cm} \pi_{x_0}h^{(i)}_j=h^{(i-1)}_{j-2^{i-1}},\quad  i\geq
2,\quad
2^{i-1}+2^{i-2}+1\leq j\leq2^i,\\

\vspace{.2cm} \pi_{x_1}h^{(i)}_j=h^{(i)}_j, \quad i\geq 1, \quad
1\leq
j\leq2^{i-1},\\

\vspace{.2cm}
\pi_{x_1}h^{(i)}_j=h^{(i+1)}_{j+2^{i-1}},\quad  i\geq
2,\quad
2^{i-1}+1\leq j\leq2^{i-1}+2^{i-2},\\

\pi_{x_1}h^{(i)}_j=h^{(i)}_{j-2^{i-3}},\quad i\geq
3,\\
\vspace{.2cm}
\hspace{3cm} 2^{i-1}+2^{i-2}+1\leq j\leq2^{i-1}+2^{i-2}+2^{i-3},\\

\vspace{.2cm}
\pi_{x_1}h^{(i)}_j=h^{(i-1)}_{j-2^{i-1}},\quad  i\geq
3,\quad
2^{i-1}+2^{i-2}+2^{i-3}+1\leq j\leq2^i.\\
\end{array}
\end{equation}

There is a one-to-one correspondence $\psi$ between
$B\setminus\{h^{(0)}\}$ and the set of all dyadic rationals from the
interval $(0,1)$ given by
$\psi(h^{(i)}_j)=\frac{j-1}{2^i}+\frac1{2^{i+1}}$, that is, each
basis function corresponds to the point of its biggest jump (where
the function changes the sign).

Below we will use the following simple observation, which can also
be used to derive equalities~\eqref{eqn_actionL2}. If a function
$h(x)\in L_2([0,1])$ changes its sign at the point $x_0$ then for
any $g\in F$ the function $(\pi_gh)(x)$ changes its sign at the
point $g(x_0)$. This enables us to find the image of $h^{(i)}_j,
i\geq3$ under action of $\pi_{x_k}$, $k=0,1$ in the following easy
way:
$$\pi_{x_k}h^{(i)}_j=\psi^{-1}\bigl(x_k(\psi(h^{(i)}_j))\bigr)$$
In other words the following diagram is commutative for $k=0,1$

\[\begin{CD}
h^{(i)}_j @>{\pi_{x_k}}>> h^{(i')}_{j'}\\
@V{\psi}VV               @V{\psi}VV\\
\frac{j-1}{2^i}+\frac1{2^{i+1}} @>>x_k>
\frac{j'-1}{2^{i'}}+\frac1{2^{i'+1}}
\end{CD}
\] \vspace{.5cm}

Now we define the Schreier graph of the action of a group on a
Hilbert space.

\noindent Let $\mathcal H$ be a Hilbert space with an orthonormal
basis $\{h_i, i\geq1\}$. Suppose there is a representation $\pi$ of
a group $G=\langle S\rangle$ in the space of all bounded linear
operators $\mathcal B(\mathcal H)$. We denote the image of $g\in G$
under $\pi$ as $\pi_g$.

\begin{defin}
The \emph{Schreier graph $\Gamma$ of the action of a group $G$ on a
Hilbert space $H$ with respect to the basis $\{h_i,i\geq1\}$ of $H$
and generating set $S\subset G$} is an oriented labelled graph
defined as follows. The set of vertices of $\Gamma$ is the basis
$\{h_i,i\geq1\}$ and there is an arrow from $h_i$ to $h_j$ with
label $s\in S$ if and only if $\langle\pi_s(h_i),h_j\rangle\neq0$
(in other words the coefficient of $\pi_s(h_i)$ at $h_j$ in the
basis $\{h_i,i\geq1\}$ is nonzero).
\end{defin}

The argument above shows that the Schreier graph of the Thompson's
group action on $L_2([0,1])$ with respect to the Haar basis and
generating set $\{x_0,x_1\}$ coincides modulo a finite part with the
Schreier graph $\Gamma(F,\{x_0,x_1\},D)$. In order to complete the
picture we have to find the images under the action of $\pi_{x_0}$
and $\pi_{x_1}$ of those $h^{(i)}_j$ which are not listed
in~\eqref{eqn_actionL2}.

Again straightforward computations give the following equalities.
\vspace{.3cm}

$\pi_{x_0}h^{(0)}=(\frac14+\frac{\sqrt2}2)h^{(0)} - \frac14h^{(0)}_1 +
(-\frac12+\frac{\sqrt2}4)h^{(1)}_1,$

$\pi_{x_0}h^{(0)}_1=\frac14h^{(0)} + (-\frac14+\frac{\sqrt2}2)h^{(0)}_1 +
(\frac12+\frac{\sqrt2}4)h^{(1)}_1,$

$\pi_{x_0}h^{(1)}_2=(\frac12-\frac{\sqrt2}4)h^{(0)} + (\frac12+\frac{\sqrt2}4)h^{(0)}_1 -
\frac12h^{(1)}_1,$ \vspace{.3cm}

$\pi_{x_1}h^{(0)}=(\frac58+\frac{\sqrt2}4)h^{(0)} + (-\frac38+\frac{\sqrt2}4)h^{(0)}_1
-\frac{\sqrt2}8h^{(1)}_2+(\frac14-\frac{\sqrt2}4)h^{(2)}_3,$

$\pi_{x_1}h^{(0)}_1=(-\frac38+\frac{\sqrt2}4)h^{(0)} + (\frac58+\frac{\sqrt2}4)h^{(0)}_1
-\frac{\sqrt2}8h^{(1)}_2+(\frac14-\frac{\sqrt2}4)h^{(2)}_3,$

$\pi_{x_1}h^{(1)}_1=\frac{\sqrt2}8h^{(0)} + \frac{\sqrt2}8h^{(0)}_1
+(-\frac14+\frac{\sqrt2}2)h^{(1)}_2+(\frac12+\frac{\sqrt2}4)h^{(2)}_3,$

$\pi_{x_1}h^{(2)}_4=(-\frac14+\frac{\sqrt2}4)h^{(0)} + (-\frac14+\frac{\sqrt2}4)h^{(0)}_1
+(\frac12+\frac{\sqrt2}4)h^{(1)}_2-\frac12h^{(2)}_3.$

These computations together with Proposition~\ref{shreier} prove the following proposition.

\begin{prop}
The Schreier graph of Thompson's group action on $L_2([0,1])$ with
respect to the Haar basis and the generating set $\{x_0,x_1\}$ has
the following structure (dashed arrows are labelled by $x_0$ and
solid arrows by $x_1$)
\begin{center}
\epsfig{file=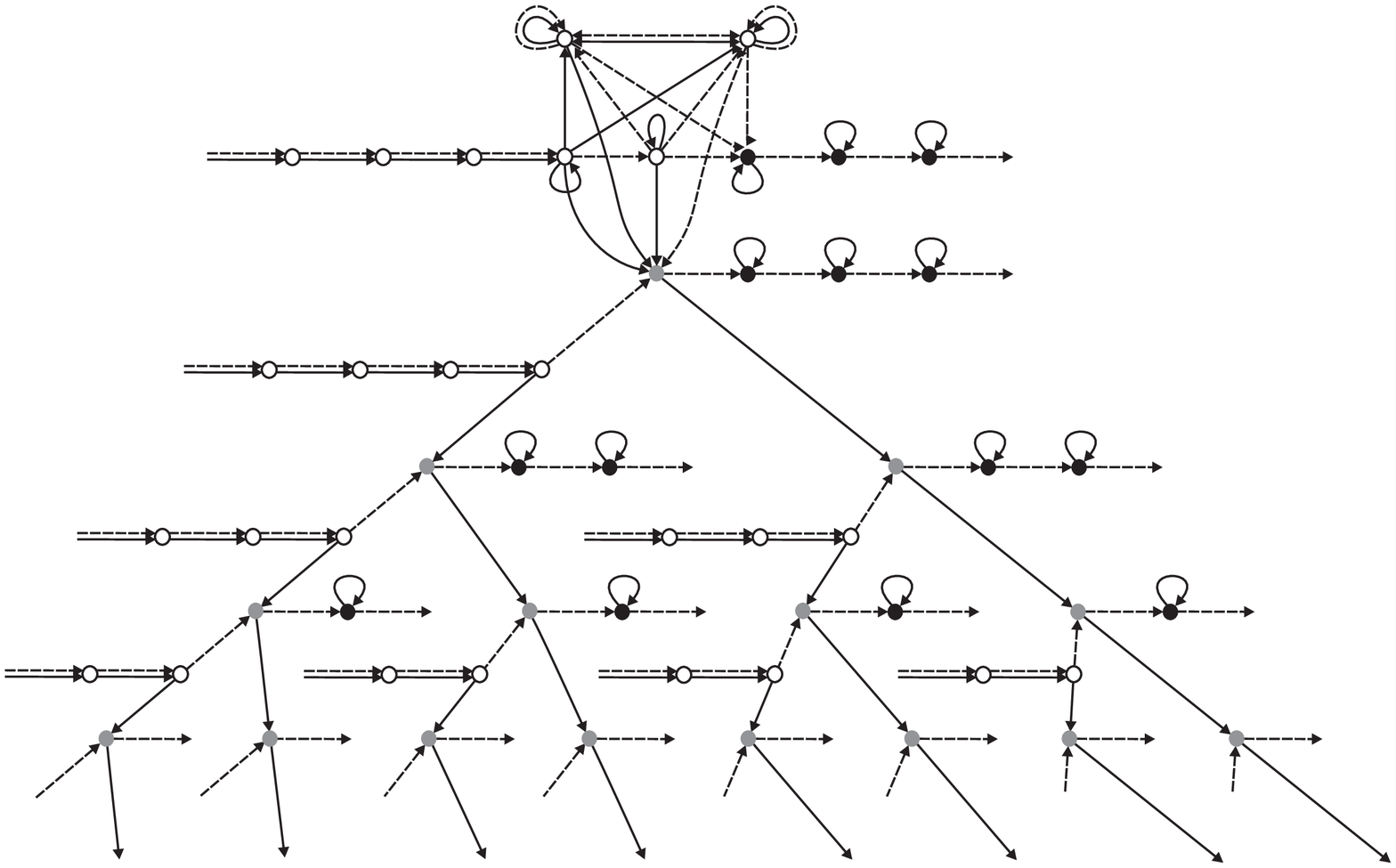,height=170pt}
\end{center}
\end{prop}

\section{Parts of the Cayley graph of $F$}
\label{sec_cayley}

Recall, that the positive monoid $P$ of $F$ is the monoid generated
by all generators $x_i$, $i\geq 0$. As a monoid it has a
presentation
\[P\cong\langle x_0,x_1,x_2,\ldots\ |\ x_kx_n=x_{n+1}x_k,\ 0\leq
k<n\rangle,\] which coincides with the infinite
presentation~\eqref{eqn_presF} of $F$. The group $F$ itself can be
defined as a group of left fractions of $P$ (i.e. $F=P^{-1}\cdot
P$).

It was shown in~\cite{grigorch_s:amenab} (see
also~\cite{grigorch:growth_amenab}) that the amenability of $F$ is
equivalent to the right amenability (with respect to our
convention~\eqref{eqn_conv}) of $P$. Moreover, let $\Gamma_F$ be the
Cayley graph of $F$ with respect to the generating set $\{x_0,x_1\}$
and $\Gamma_P$ be the induced subgraph of $\Gamma_F$ containing
positive monoid $P$. The following proposition is of a folklore
type.

\begin{prop}
Amenability of $F$ is equivalent to amenability of the graph
$\Gamma_P$.
\end{prop}

\begin{proof}
Any finite set $T$ in $F$ can be shifted to the positive monoid $P$,
i.e. there is some $g\in F$ such that $Tg\subset P$. The boundary
$\partial_P(Tg)$ of this shifted set in $\Gamma_P$ is not bigger
than the boundary of $T$ in $\Gamma_F$. Hence, Cheeger constant of
$\Gamma_P$ is not bigger than the one of $\Gamma_F$. Thus,
non-amenability of $\Gamma_P$ implies non-amenability of $F$.

Suppose that $\Gamma_P$ is amenable. Then for any $\varepsilon>0$
there exists a subset $T$ of $P$, such that its boundary
$\partial_PT$ in $\Gamma_P$ satisfies
\begin{equation}
\label{eqn_cheeger}
\frac{|\partial_PT|}{|T|}<\frac{\varepsilon}4
\end{equation}
Now we can bound the size of the boundary $\partial_FT$ of $T$ in
$\Gamma_F$. We use simple observations that for finite sets $A$ and
$B$ of the same cardinality $|A\setminus B|=|B\setminus
A|=\frac12|A\Delta B|$ and that $|Tx_i^{-1}\Delta
T|=|(Tx_i^{-1}\Delta T)x_i|=|Tx_i\Delta T|$.

We have
\[\partial_FT=(Tx_0\setminus T)\cup(Tx_1\setminus T)\cup(Tx_0^{-1}\setminus T)\cup(Tx_1^{-1}\setminus
T).\]
Therefore,
\begin{align*}
|\partial_FT|\leq |Tx_0\setminus T|+|Tx_1\setminus
T|+|Tx_0^{-1}\setminus T|+|Tx_1^{-1}\setminus T|\\
\leq\frac12(|Tx_0\Delta T|+|Tx_1\Delta
T|+|Tx_0^{-1}\Delta T|+|Tx_1^{-1}\Delta T|)\\
=|Tx_0\Delta T|+|Tx_1\Delta T|=2|Tx_0\setminus T|+2|Tx_1\setminus
T|\leq 4|\partial_PT|<\varepsilon|T|
\end{align*}
since $Tx_i\setminus T\subset\partial_PT$ for $i=1,2$ and
by~\eqref{eqn_cheeger}. This shows that $\Gamma_F$ is also amenable
in this case.
\end{proof}

In this section we explicitly construct the induced subgraph
$\Gamma_S$ of $\Gamma_F$ containing the set of vertices
\begin{equation}
\label{eqn_defn_of_S}
S=\{x_nu\ \bigl|\ n\geq0,\ u \text{ is a word over } \{x_0,x_1\}\}.
\end{equation}
We also prove that this graph is non-amenable.

Since $S$ is included in the positive monoid of $F$ and contains
elements from the infinite generating set $\{x_0,x_1,x_2,\ldots\}$,
it is natural to use the language of forest diagrams developed
in~\cite{belk:phd,sunic:tamari} (though the existence of this
representation was originally noted by
K.Brown~\cite{brown:finiteness}). First we recall the definition and
basic facts about this representation of the elements of $F$.

There is a one-to-one correspondence between the elements of the
positive monoid of $F$ and rooted binary forests. More generally,
there is a one-to-one correspondence between elements of $F$ and,
so-called, reduced forest diagrams, but for our purposes (and for
simplicity) it is enough to consider only the elements of the
positive monoid.

A \emph{binary forest} is an ordered sequence of finite rooted
binary trees (some of which may be trivial). The forest is called
\emph{bounded} if it contains only finitely many nontrivial trees.

\begin{center}
\epsfig{file=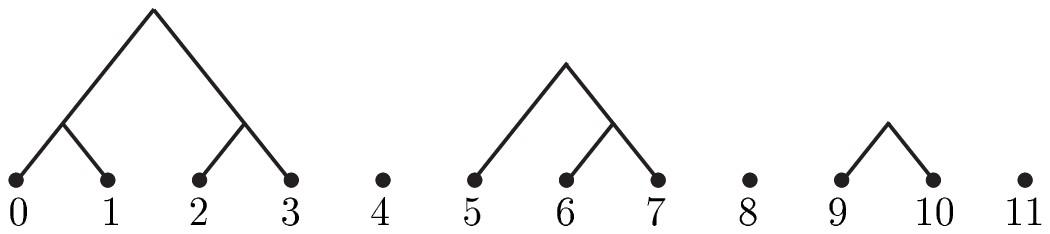}
\end{center}

There is a natural way to enumerate the leaves of the trees in the
forest from left to right. First we enumerate the leaves of the
first tree from left to right, then the leaves of the second tree,
etc. Also there is a natural left-to-right order on the set of the
roots of the trees in the forest.

The product $\mathfrak{f}\mathfrak{g}$ of two rooted binary forests
$\mathfrak{f}$ and $\mathfrak{g}$ is obtained by stacking the forest
$\mathfrak{g}$ on the top of $\mathfrak{f}$ in such a way, that the
$i$-th leaf of $\mathfrak{g}$ is attached to the $i$-th root of
$\mathfrak{f}$.

For example, if $\mathfrak{g}$ and $\mathfrak{f}$ have the following
diagrams
\begin{center}
\includegraphics{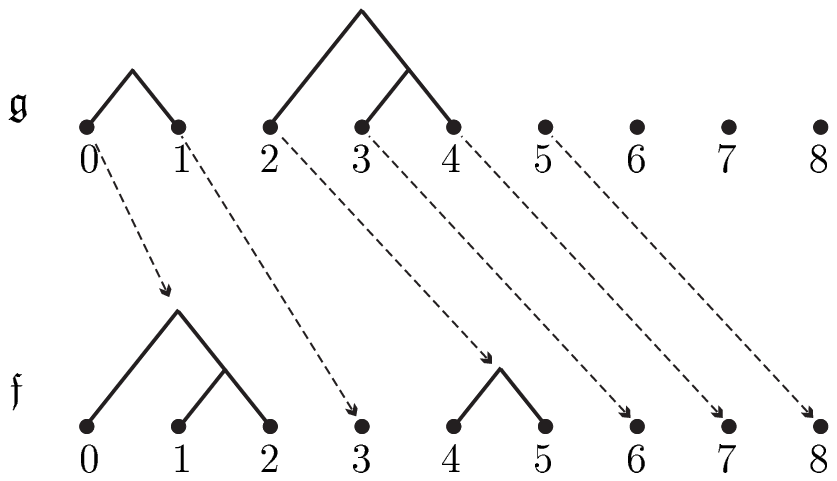}
\end{center}
then their product $\mathfrak{f}\mathfrak{g}$ is the following
rooted binary forest
\begin{center}
\includegraphics{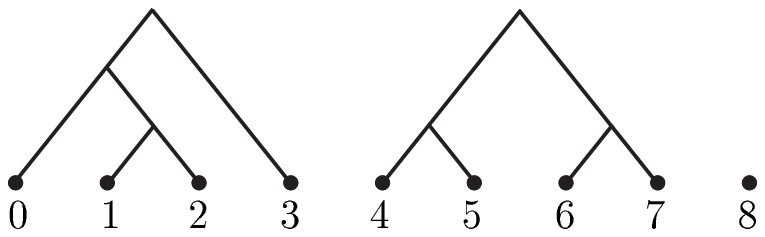}
\end{center}

With this operation the set of all rooted binary forests is
isomorphic (see~\cite{belk:phd,sunic:tamari}) to the positive monoid
of Thompson's group $F$, where $x_n$ corresponds to the forest in
which all the trees except the $(n+1)$-st one (which has number $n$)
are trivial and the $(n+1)$-st tree represents a single caret. Below
is the picture of the forest corresponding to $x_3$.

\begin{center}
\includegraphics{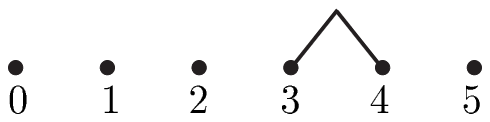}
\end{center}

The multiplication rule for the forests implies the following
algorithm for construction of the rooted forest corresponding to the
element $x_{i_1}x_{i_2}x_{i_3}\cdots x_{i_n}$ of the positive monoid
of $F$. Start from the trivial forest (where all the trees are
singletons) and consequently add the carets at the positions $i_1$,
$i_2$,\ldots, $i_n$ (counting from 0 the roots of the trees in the
forest in previous iteration).

For our main result in this section we need two lemmas.

\begin{lemma}
\label{lem_forest1} Let $u$ be a word from the positive monoid of the form
$u=x_nv$, where $n\geq 2$ and $v$ is a word over the alphabet
$\{x_0,x_1\}$ of length at most $n-2$. Then this word is not equal
in $F$ to any other word of the form $x_mw$, where $w$ is a word
over $\{x_0,x_1\}$.
\end{lemma}

\begin{proof}
The forest diagram corresponding to $u$ has a caret $c$ connecting
the $n$-th and $(n+1)$-st leaves corresponding to $x_n$ and possibly
some nontrivial trees to the left of $c$.

\begin{figure}[h]
\begin{center}
\includegraphics{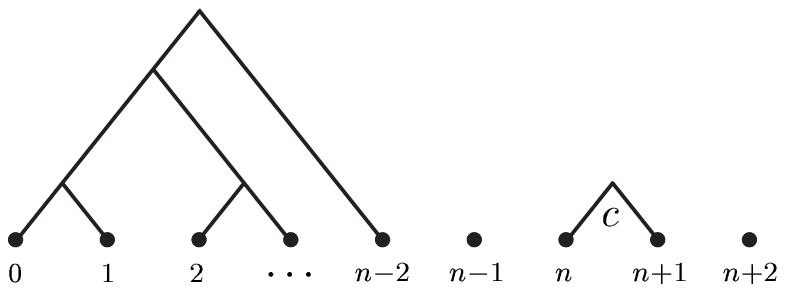}
\end{center}
\caption{Forest corresponding to $x_nv$\label{fig_forest2_1}}
\end{figure}

Indeed, after attaching the caret corresponding to $x_n$ all the
other carets are attached at positions either $0$ or $1$. Each of
these carets decreases the number of trees to the left of caret $c$
by 1. Since originally there were $n$ trees to the left from $c$ and
the length of $v$ is at most $n-2$, there must be at least $2$ trees
to the left of $c$ in the forest representing $u$.

Suppose there is another word of the form $x_mw$ in the positive
monoid of $F$ whose corresponding rooted forest coincides with the
forest of $u$. Since there are at least $2$ trees to the left of
caret $c$ one can not obtain this caret by applying $x_0$ or $x_1$.
Therefore it was constructed at the first step with application of
$x_m$. Thus $x_m=x_n$ because this caret connects the $n$-th and
$(n+1)$-st leaves, which, in turn, implies that $v=w$ in $F$. But
both $v$ and $w$ are the elements of a free submonoid generated by
$x_0$ and $x_1$, yielding that $x_nv=x_mw$ as words.
\end{proof}

\begin{lemma}
\label{lem_forest2} Let $u$ be a word from the positive monoid of the form
$u=x_nvx_1v'$, where $n\geq2$ and $v$ is a word over the alphabet
$X=\{x_0,x_1\}$ of length $n-2$. Then this word is not equal in $F$
to any other word of the form $x_mw$, where $w$ is a word over
$\{x_0,x_1\}$.
\end{lemma}

\begin{proof}
The rooted forest corresponding to $x_nv$ is constructed in
Lemma~\ref{lem_forest1} and shown in Figure~\ref{fig_forest2_1}.
Note, that there are exactly 2 trees (one of which is shown trivial
in Figure~\ref{fig_forest2_1}) to the left of caret $c$. At the next
step we apply generator $x_1$, which attaches the new caret $d$ that
connects the root of the second of these trees to the root of caret
$c$. The resulting forest is shown in Figure~\ref{fig_forest2_2}.

\begin{figure}[h]
\begin{center}
\includegraphics{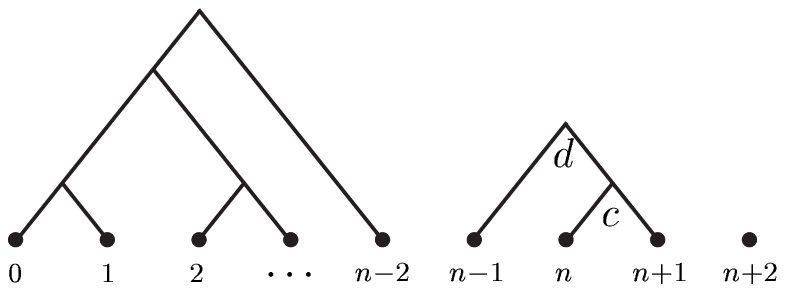}
\end{center}
\caption{Forest corresponding to $x_nvx_1$\label{fig_forest2_2}}
\end{figure}

Next, applying $v'$ adds some extra carets on top of the picture.
The final rooted forest is shown in Figure~\ref{fig_forest2_3}.

\begin{figure}[h]
\begin{center}
\includegraphics{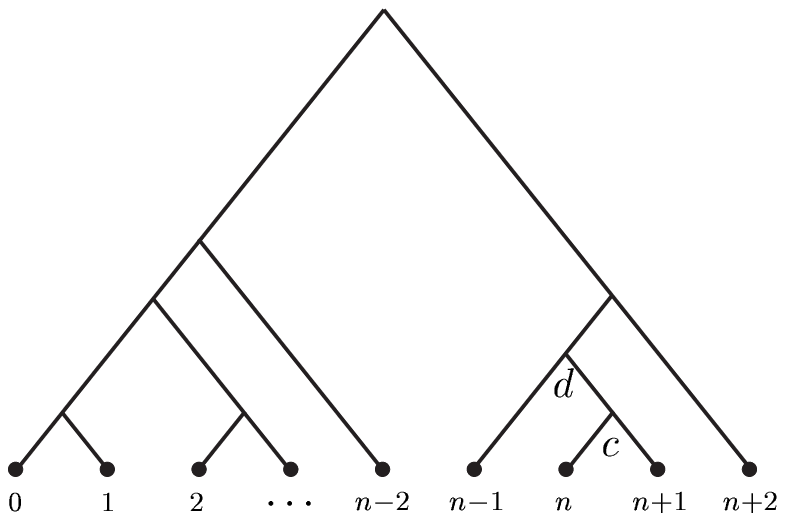}
\end{center}
\caption{Forest corresponding to $x_nvx_1v'$\label{fig_forest2_3}}
\end{figure}

Analogously to Lemma~\ref{lem_forest1} we obtain that if the rooted
forest of $x_mw$ coincides with the one of $u$, the caret $c$ could
appear only from the initial application of $x_m$ (since it must be
placed before caret $d$ is placed). Hence $x_n=x_m$ and $v=w$ as
words, because the submonoid generated by $x_0$ and $x_1$ is free.
\end{proof}

Let $\Gamma_S$ be the induced subgraph of the Cayley graph
$\Gamma_F$ of $F$ that contains all the vertices of from the set $S$
(recall the definition of $S$ in~\eqref{eqn_defn_of_S}). As a direct
corollary of Lemma~\ref{lem_forest1} and Lemma~\ref{lem_forest2}, we
can describe explicitly the structure of $\Gamma_S$ (see
Figure~\ref{fig_cayley}, where solid edges are labelled by $x_1$ and
dashed by $x_0$).

\begin{prop}
\label{prop_structure}
The structure of $\Gamma_S$ is as follows
\begin{itemize}
\item[$(a)$] $\Gamma_S$ contains the infinite binary tree $T$ corresponding to the free
submonoid generated by $x_0$ and $x_1$;
\item[$(b)$] for each $n\geq2$ there is a binary tree $T_n$ in $\Gamma_S$ consisting
of $n-2$ levels which grows from the vertex $x_n$ and does not
intersect anything else;
\item[$(c)$] Each vertex $x_nv$ of the boundary of $T_n$ (i.e. $v$ has length
$n-2$) has two neighbors $x_nvx_1$ and $x_nvx_0$ outside $T_n$. The
first one is the root of an infinite binary tree which does not
intersect anything else. The second one coincides with the vertex
$vx_0x_1$ of the binary tree $T$.
\end{itemize}
\end{prop}

\begin{figure}[h]
\begin{center}
\includegraphics{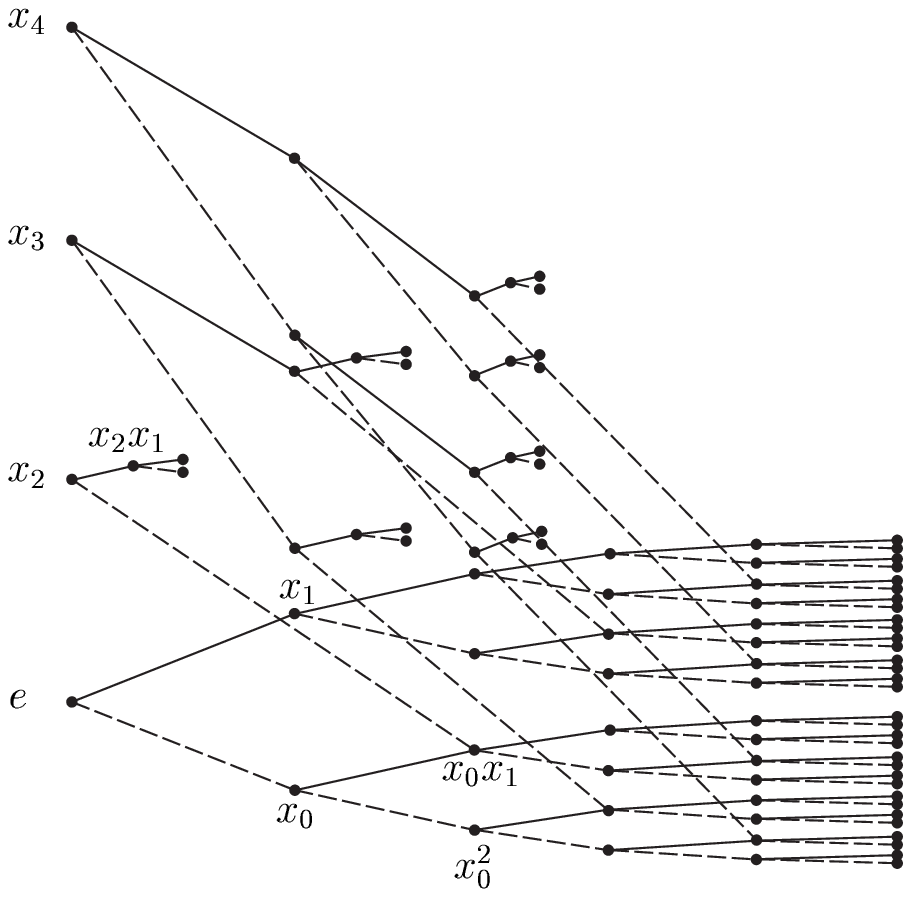}
\end{center}
\caption{Induced subgraph $\Gamma_S$ of the Cayley graph of
$F$\label{fig_cayley}}
\end{figure}

\begin{prop}
\label{nonamenab}
The graph $\Gamma_S$ is non-amenable.
\end{prop}

In order to prove this Proposition we will use equivalent to the
amenability doubling condition (or Gromov doubling condition)
\cite{harpe_cg:paradoxical}.

\begin{theorem}[Gromov's Doubling Condition]
\label{thm_gromov}
Let $X$ be a connected graph of bounded degree. Then $X$ is
non-amenable if and only if there is some $k\geq1$ such that for any
finite nonempty subset $S\subset V(X)$ we have
$$|\mathcal N_k(S)|\geq2|S|,$$
where $\mathcal N_k(S)$ is the set of all vertices $v$ of $X$ such
that $d_X(v,S)\leq k$.
\end{theorem}

\begin{proof}[Proof of Proposition \ref{nonamenab}]
In order to use the Theorem~\ref{thm_gromov} it is enough to
construct two injective maps $f,g: V(X)\to V(X)$ with distinct
images, that do not move vertices farther than by distance $k$.

For any vertex $x_nv$ in $S$ put
$$f(x_nv)=x_nvx_1x_0,$$
$$g(x_nv)=x_nvx_1x_1.$$

For any vertex $x_nv$ of $S$ we have $d(x_nv,f(x_nv))=2$ and
$d(x_nv,g(x_nv))=2$, so the last condition of
Theorem~\ref{thm_gromov} is satisfied.

The relation $f(x_nv)=f(x_mw)$ implies $x_nvx_1x_0=x_mwx_1x_0$ and
$x_nv=x_mw$. Hence $f$ is an injection. The same is true for $g$.

Now suppose $f(x_nv)=g(x_mw)$ or, equivalently,
\begin{equation}
\label{eqn_glue}
x_nvx_1x_0=x_mwx_1x_1
\end{equation}

The words $x_nvx_1$ and $x_mwx_1$ represent different vertices in
$\Gamma_S$ since otherwise we would get $x_0=x_1$. According to
Proposition~\ref{prop_structure} the equality~\eqref{eqn_glue} is
possible only in case when $x_nvx_1$ is a vertex of the boundary of
$T_n$ and $x_mwx_1$ is a vertex of $T$. But by
Proposition~\ref{prop_structure}$(c)$ in this case the vertex
$x_nvx_1x_0$ coincides with the vertex $vx_1x_0x_1$ of $T$ which can
not coincide with $x_mwx_1x_1$. Indeed, otherwise we get
\[vx_1x_0=x_mwx_1.\]
Then $vx_1$ and $x_mw$ must represent different vertices of
$\Gamma_S$. According to Proposition~\ref{prop_structure} the last
equality  may occur only in case when $vx_1$ belongs to the boundary
of tree $T_r$ for some $r\geq2$, which is not the case because
$vx_1\in T$. Therefore the equality~\eqref{eqn_glue} is never
satisfied and the images of $f$ and $g$ are distinct.

Thus by Theorem~\ref{thm_gromov} the graph $\Gamma_S$ is
non-amenable.
\end{proof}

\bibliographystyle{alpha}
\def\cprime{$'$} \def\cprime{$'$} \def\cprime{$'$} \def\cprime{$'$}
  \def\cprime{$'$}

\noindent\texttt{savchuk@math.tamu.edu}

\noindent
Department of Mathematics \\Texas A\&M University\\
College Station, TX 77843-3368 \\USA

\end{document}